\documentclass[12pt]{amsart}
\usepackage{stmaryrd} 
\usepackage{hyperref} 
\usepackage{enumerate}
 \theoremstyle{plain}
 \newtheorem*{teorema}{Theorem}
 \newtheorem{theorem}{Theorem}
  \newtheorem{lemma}{Lemma}
 \newtheorem{proposition}{Proposition}
 
 \theoremstyle{definition}

 \newtheorem{definition}{Definition}

 \newcommand{\Z}{\mathbb Z}
 \newcommand{\R}{\mathbb R}
 \newcommand{\C}{\mathbb C}

 \newcommand{\cC}{\mathcal C}

 \newcommand{\af}{\alpha}
 \newcommand{\be}{\beta}
 \newcommand{\ep}{\varepsilon}
 \newcommand{\ga}{\gamma}
 
 \newcommand{\si}{\sigma}
 \newcommand{\de}{\delta}
 
 \newcommand{\lam}{\lambda}
 \newcommand{\te}{\theta}
\newcommand{\up}{\upsilon}
 \newcommand{\om}{\omega}

 \newcommand{\dga}{\dot{\ga}}
 
 \newcommand{\dy}{\dot{y}}

\begin{document}
\title [Busemann functions for the N-body problem]
{Busemann functions for the N-body problem}
\author[Percino]{Boris Percino}
 \address{ Instituto de Matem\'aticas,  Universidad Nacional Aut\'onoma
   de M\'exico\\     M\'exico DF 04510, M\'exico.}
 \email{borispercino@yahoo.com.mx}
\author[S\'anchez-Morgado]{H\'ector S\'anchez-Morgado}

\email{hector@matem.unam.mx}
\subjclass{37J50,70F10}
\keywords{Central configuration, free time minimizer, Hamilton-Jacobi equation}
\maketitle
\begin{abstract}
Following ideas in \cite{MV}, we prove that the Busemann function of
the parabolic homotetic motion for a minimal central coniguration of
the N-body problem is a viscosity solution of the Hamilton-Jacobi
equation and that its calibrating curves are asymptotic to
the homotetic motion. 
\end{abstract}

\section{Introduction}
\label{intro}
We consider the N-body problem with potential $U:(\R^d)^N\to]0,\infty]$
\[U(x)=\sum_{i<j}\frac{m_im_j}{r_{ij}}\]
where $x=(r_1,\ldots,r_N)\in(\R^d)^N$ is a configuration of $N$
punctual positive masses $m_1,\ldots,m_N$ in Euclidean space $\R^d$,
and $r_{ij}=|r_i-r_j|$. We will adopt the variational point of view and
consider the Lagrangian $L:(\R^d)^{2N}\to]0,\infty]$
\[L(x,v)=T(v)+U(x)=\frac 12\sum_{i=1}^Nm_i|v_i|^2+U(x)\]
as well as the action of an absolutely continuous curve $\ga:[a,b]\to(\R^d)^N$
given by
\[A_L(\ga)=\int_a^bL(\ga(t),\dga(t))dt\]
taking values in $]0,\infty]$. We will denote by $\cC(x,y,\tau)$
the set of curves binding two given
configurations $x,y\in(\R^d)^N$ in time $\tau>0$,
that is to say,
\[\cC(x,y,\tau)=
\{\ga:[a,b]\to(\R^d)^N \text{ absolutely continuous }
\vert b-a=\tau,\,\ga(a)=x, \ga(b)=y\},\]
and $\cC(x,y)$ will denote the set of curves
binding two configurations $x,y\in(\R^d)^N$
without any restriction on time,
\[\cC(x,y)= \bigcup_{\tau>0}\cC(x,y,\tau)\;.\]
In what follows we will consider curves
which minimize the action on these sets.
We define the function
$\phi:(\R^d)^N\times (\R^d)^N\times (0,+\infty)\to\R$,
\[\phi(x,y;\tau)=
\inf\{A(\ga)\vert\ga\in\cC(x,y,\tau)\},\]
and the {\em Ma\~n\'e critical action potential}
\[\phi(x,y)=\inf\{A(\ga) \vert \ga\in\cC(x,y)\}=
\inf\{\phi(x,y,\tau) \vert \tau>0\}.\]
defined on $(\R^d)^N\times (\R^d)^N$.
In the first definition, the infimum is achieved
for every pair of configurations
$x,y\in (\R^d)^N$. In the second one the
infimum is achieved if and only if $x\neq y$.
These facts are essentially due
to the lower semicontinuity of the action.
\begin{definition}
A \textsl{free time minimizer}
defined on an interval $J\subset\R$
is an absolutely continuous curve
$\ga:J\to (\R^d)^N$ which satisfies
$A(\ga|_{[a,b]})=\phi(\ga(a),\ga(b))$
for all compact subinterval $[a,b]\subset J$.
\end{definition}

There is a relatively easy way to give an example
of a free time minimizer defined on an unbounded interval via
the minimal configurations of the problem. 
Recall that the moment of inertia (about the origin)
of a given configuration $x\in (\R^d)^N$ is
\[I(x)=\sum_{i=1}^Nm_i|r_i|^2.\]
We will use the norm in $(\R^d)^N$ given by $\|x\|^2=I(x)$.
We say that $x_0\in (\R^d)^N$ with $I(x_0)=1$ is a (normal) minimal
configuration of the problem if $U(x_0)=\min\{U(x)\vert x\in (\R^d)^N, I(x)=1\}$.
Also recall that a central configuration is a configuration $a\in (\R^d)^N$
which admits homothetic motions i.e. of the form $x(t)=\lam(t)a$.
This happens if and only if $a$ is a critical point of 
$x\mapsto\|x\|U(x)$ and $\lam$ satisfies the Kepler equation
$\ddot\lam\lam^{2}=-U(a)/I(a)$.
Thus minimal configurations are in particular central configurations.
For a given central configuration $a$, choosing $\mu$ that satisfies
$\mu^3=U(a)/I(a)$ we have that $x(t)= \mu t^{2/3}a$ is a parabolic homothetic 
motion. By Proposition 19 in \cite{DM}, if $a$ is a minimal configuration 
such motions are free time minimizers, and we don't know if there are
other central configurations with this property. 

Our main result is 
\begin{teorema}\label{Busemann}
Let $x_0$ be a minimal central configuration with $\|x_0\|=1$, $U(x_0)=U_0$, 
and consider the the homotetic motion
$\ga_0(t)=ct^{\frac 23}x_0$, $c=(\frac 92 U_0)^{\frac 13}$.
Then the Busemann function
 \[u(x) =\sup_{t>0} \left[\phi(0,\ga_0(t))-\phi(x,\ga_0(t))\right]
       =\lim_{t\to+\infty}\left[\phi(0,\ga_0(t))-\phi(x,\ga_0(t))\right]\]
is a viscosity solution of the Hamilton-Jacobi equation
\begin{equation}
  \label{eq:hj}
  \|Du(x)\|^2=2U(x).
\end{equation}
Moreover, for any $x\in(\R^d)^N$ there is a calibrating curve
$\af:[0,\infty)\to(\R^d)^N$ of $u$ with $\af(0)=x$ and  
\begin{equation}
  \label{asintotica}
\lim_{t\to\infty}\|\af(t)  t^{-\frac 13}-cx_0\|=0. 
\end{equation}
\end{teorema}
\section{Preliminaries}
\label{sec:pre}
We start recalling two Theorems of E. Maderna
\begin{theorem}\cite{M}\label{Mad1}
  There are constants  $\af, \be > 0$ such that for all $T > 0$,
\[ \phi(x,y;T)\le \af\frac{R^2}T +\be\frac TR\]
whenever x and y are contained in a ball of radius $R >0$ 
in $(\R^d)^N$. 
\end{theorem}
\begin{theorem}\cite{M}\label{Mad2}
  There is $\eta>0$ such that for all $y,z\in(\R^d)^N$
\[\phi(y,z)\le\eta\|y-z\|^{\frac 12}\] 
\end{theorem}

\begin{proposition}\label{potential-line}
Denote by $S(a,b)$ the action potential for the one-dimensional Kepler's 
problem with potential energy $U_0/r$ and  by $S(a,b;t)$ the action of
the only solution to Kepler's problem on the half line that  goes from $a$
to $b\ge a$ in time $t$. 
\begin{enumerate}[(a)]
\item The function
\[G(r)=S(0,r;1)-S(0,r)=S(0,r;1)-(8U_0r)^{\frac 12}\]
is decreasing on $]0,c[$, increasing on $]c,\infty[$,
$G(c)=G'(c)=0$, $G''(c)=5/3$. It follows that if $\ep>0$ is sufficiently 
small, there is $\de(\ep)>0$ such that $|r-c|\le\ep$ if $G(r)\le\de(\ep)$.
\item For $\bar\ep>0$ we have
  \begin{equation}\label{S(r,1+e)}
    S(0,r;1+\ep)=\frac{r^2}{2(1+\ep)}+o(r^2)
  \end{equation}
as $r\to\infty$ uniformly on $\ep\in[0,\bar\ep]$. 
\item As $\si\to 1, s\to\infty$ 
  \begin{equation}\label{eq:S(r,s)}
S(r,cs^{\frac 23},\si s)=(6U_0^2s)^{\frac 13}(2+\frac 59(\si-1)^2+o((\si-1)^2))
-(8U_0r)^{\frac 12}+O(s^{-\frac 13})  
  \end{equation}
uniformly on $r\in[0,s^{1/3}]$. Thus
\begin{equation}\label{eq:S(r,u)}
S(r,u,\frac \si3\Bigl(\frac{2u^3}{U_0}\Bigr)^{\frac 12})=(8U_0)^{\frac 12}(u ^{\frac 12}(1+\frac 5{18}(\si-1)^2+o((\si-1)^2))-r^{\frac 12})
+O(u^{-\frac 12})  
  \end{equation}
as $\si\to 1, u\to\infty$ uniformly on $r\in[0,u^{1/2}]$.
\end{enumerate}
\end{proposition}
\begin{proof}
Items (a), (b) are proved in \cite{MV}. We prove item (c) following also
\cite{MV}. Let $h(r,s,\si)$ be the energy of the only solution that goes
from $r$ to $cs^{\frac 23}$ in time $\si s$. Then
\[\si s=\int_r^{cs^{\frac 23}}\frac{du}{\sqrt{2(h+U_0/u)}}=
\frac {3s}2\int_{rs^{-\frac 23}/c}^1\frac{dv}{\sqrt{cs^{\frac 23}h/U_0+1/v}}\]
Define
\[F(x,y,k)=\int_{x^2}^1\sqrt{\frac v{1+kv}}dv-\frac 23(1+y)\]
 then 
\[F(\Bigl(\frac rc\Bigr)^{\frac 12}s^{-\frac 13},\si-1,\frac c{U_0}s^{\frac 23}h(r,s,\si ))=0.\]
Since $F(0,0,0)=0$, $F_k(0,0,0)=-\frac 15$, the implicit function
theorem can be applied to solve $F(x,y,k(x,y))=0$. Implicit
differentiation gives 
\begin{align}\label{k}
  k(x,y)&=-\frac{10}3y+\frac{125}{21}y^2+o(x^2+y^2)\\\nonumber
h(r,s,\si )&=\Bigl(\frac 29 U_0^2\Bigr)^{\frac 13}s^{-\frac 23}k(\Bigl(\frac rc\Bigr)^{\frac 12}s^{-\frac 13},\si -1)\\\label{energy}
&=-\frac{10U_0}{3c}s^{-\frac 23}((\si -1)+\frac{125}{21}(\si -1)^2
+o(s^{-\frac 23}+(\si -1)^2))
\end{align}
\begin{align}\nonumber
S(r,cs^{\frac 23},\si s)&=\int_r^{cs^{\frac 23}}
\frac{h+2U_0/u}{\sqrt{2(h+U_0/u)}}du\\\nonumber
&=\int_r^{cs^{\frac 23}}\sqrt{2(h+\frac{U_0}u)}du-h\si s\\\nonumber
&=(6U_0^2s)^{\frac 13}A(x,k)-\Bigl (\frac 29 U_0^2s\Bigr)^{\frac 13}k\si 
\end{align}
where $k$ is given by \eqref{k}, with $x=(r/c)^{\frac 12}s^{-\frac 13}$,  
$y=\si-1$, and
\begin{align*}
A(x,k)=\int_{x^2}^1\sqrt{k+\frac 1v}dv&=A_0(k)-B(x,k),\\
A_0(k)=\int_0^1\sqrt{k+\frac 1v}dv, & 
\quad B(x,k)=\int_0^{x^2}\sqrt{k+\frac 1v}dv
\end{align*}
We have as in \cite{MV}
\[A_0(k)=2+\frac k3-\frac {k^2}{20}+o(k^2), B(x,k)=2|x|+O(k|x|^3).\]
Thus
\begin{align*}
S(r,cs^{\frac 23},\si s)&=(6U_0^2s)^{\frac 13}
(2+\frac k3(1-\si)-\frac{k^2}{20}+o(k^2) +O(k|x|^3)-2|x|)\\ 
&=(6U_0^2s)^{\frac 13}(2+\frac 59(\si -1)^2+o(s^{-\frac 23}+(\si -1)^2)
+O(\frac{\si -1}s))-(8U_0r)^{\frac 12}
\end{align*}
\end{proof}
\section{Proof of the result}
\label{sec:proof}
Consider the Kepler's problem on $(\R^d)^N$ with Lagrangian 
\[L_0(x,v)=\frac{\|v\|^2}2+\frac{U_0}{\|x\|}\]
and action potential 
\[\phi_0(x,y)=\inf_{T>0}\phi_0(x,y;T).\]
We have 
\begin{align}\label{remark1}
S(\|x\|,\|y\|;T)&\le\phi_0(x,y;T)\le\phi(x,y;T),\\
S(0,\|x\|;T)&=\phi_0(0,x;T)=\phi(0,\|x\|x_0;T). \label{remark2}
\end{align}
\begin{lemma}\label{subsol}
Busemann function $u$ is well defined and a viscosity subsolution 
of the Hamilton-Jacobi equation \eqref{eq:hj}
\end{lemma}
\begin{proof}
   We start by showing that the function
 \[\de(t)=\left[\phi(0,\ga_0(t))-\phi(x,\ga_0(t))\right]\]
is increasing.
 If $s<t$, then
 \begin{align*}
 \de(s)-\de(t)&=\phi(x,\ga_0(t))-\phi(x,\ga_0(s))
              +\big[\,\phi(0,\ga_0(s))-\phi(0,\ga_0(t))\,\big]\\
              &=\phi(x,\ga_0(t))-\phi(x,\ga_0(s))-\phi(\ga_0(s),\ga_0(t))\\
              &\le 0,   
 \end{align*}
 where the  last inequality follows from the triangle
 inequality applied to the triple $(x,\ga_0(s),\ga_0(t))$.
 By the triangle inequality, $\de(t)\le \phi(\ga_0(0),x)$,
 hence  $\lim\limits_{t\uparrow\infty}\de(t)=\sup\limits_{t>0}\de(t)$ 
 and this limit is finite.
 
 Since
 \begin{align*}
 u(y)&=\sup_{t>0}\phi(0,\ga_0(t))-\phi(y,\ga_0(t)) \\
     &\ge\sup_{t>0}\phi(0,\ga_0(t))-\phi(y,x)-\phi(x,\ga_0(t)) \\    
     &=u(x)-\phi(y,x),
 \end{align*}
$u$ is a viscosity subsolution.
\end{proof}
Let $x\in(\R^d)^N$. By Theorem 6 in \cite{DM}, for
$cT^{2/3}>\|x\|$  there is
$y_T:[0,\tau_T]\to(\R^d)^N$ with $y_T(0)=x$, $y_T(\tau_T)=\ga_0(T)$ such that 
\begin{equation}
  \label{eq:minimiza}
A_L(y_T)=\phi(x,\ga_0(T)).  
\end{equation}

\begin{proposition}
$\lim\limits_{T\to\infty}\dfrac T{\tau_T}=1$.
\end{proposition}
\begin{proof}
From \eqref{eq:minimiza}, for all $0<t<T$ we have 
 \begin{align}\label{1}
    A_L\big(y_T\vert_{[0,t]}\big) + \phi(y_T(t),\ga_0(T))
    = \phi(x,\ga_0(T)).
 \end{align}
We know that 
\begin{align}\label{2}
 \phi(x,\ga_0(T))&\le\phi(x,0)+\phi(0,\ga_0(T))=
\phi(x,0)+2(6U_0^2T)^{1/3}\\
(\tfrac 92U_0)^{1/3}T^{2/3}-\|x\|&
\le\|x-\ga(T)\|\le\int_0^{\tau_T}\|\dy_T\|\le\sqrt{\tau_T}  
\left(\int_0^{\tau_T}\|\dy_T\|^2\right)^{1/2}\\ \notag
&\le\sqrt{2\tau_TA_L(y_T)}
\le\sqrt{2\tau_T}(\phi(x,0)+2(6U_0^2T)^{1/3})^{1/2}.
\end{align}
Thus 
\[\limsup_{T\to\infty}\frac T{\tau_T}\le\frac 83.\] 
From \eqref{remark1} and \eqref{2}
\begin{align*}
S(\|x\|,cT^{\frac 23};\tau_T)&\le\phi(x,\ga_0(T);\tau_T)
\le \phi(x,0)+2(6U_0^2T)^{1/3}\\ 
S\Bigl(\frac{\|x\|}{\tau_T^{\frac 23}},
c\Bigl(\frac T{\tau_T}\Bigr)^{\frac 23};1\Bigr)&\le
\frac{\phi(x,0)}{\tau_T^{\frac 13}}+2\Bigl(\frac{6U_0^2T}{\tau_T}\Bigr)^{\frac 13}
\end{align*}
Consider a sequence $T^j\to\infty$ such that
$\dfrac{T^j}{\tau_{T^j}}\to s$. Then 
\[S(0,cs^{\frac 23};1)\le 2(6U_0^2s)^{\frac 13}=(8U_0cs^{\frac 23})^{\frac 12}.\]
By Proposition \ref{potential-line}(a), $cs^{\frac 23}=c$ and so $s=1$.
\end{proof}

Let $\be:[0,1]\to(\R^d)^N$ be a curve joining $0$ and $x$ and
$t\le\tau_T$, then 
\begin{align}\label{define}
  \phi(0,y_T(t);t+1)&\le A_L(\be)+A_L(y_T|[0,t])\\\label{potencial}
A_L(y_T)=\phi(x,\ga_0(T))&\le A_L(\be)+\phi(0,\ga_0(T))\\
\label{combinacion}
\phi(0,y_T(t);t+1)+A_L(y_T|[t,\tau_T])&\le 2A_L(\be) +\phi(0,\ga_0(T))\\
\label{minima}\phi(0,y_T(t);t+1)+\phi(y_T(t),\ga_0(T);\tau_T-t)
&\le 2A_L(\be) +\phi(0,\ga_0(T))
\end{align}
Inequalities \eqref{define} and \eqref{potencial} give
\eqref{combinacion}, which can be written as \eqref{minima}. 

Letting $s=s(T,t)=T/t$, $\si=\si(T,t)=(\tau_T-t)/T$, from
\eqref{minima} we have  
\begin{equation}\label{homo}
  \phi(0,y_T(t)t^{-\frac 23};1+1/t)
+\phi(y_T(t)t^{-\frac 23},\ga_0(s);\si s) 
\le 2t^{-\frac 13}A_L(\be) +\phi(0,\ga_0(s)).
\end{equation}

Defining $r_T(t)={\|y_T(t)\|}{t^{-\frac 23}}$, from \eqref{remark1} and \eqref{remark2} we get 
\begin{align} \label{cero}
\phi_0(0,y_T(t)t^{-\frac 23};1+1/t)
+\phi_0(y_T(t)t^{-\frac 23},\ga_0(s);\si s) 
&\le 2t^{-\frac 13}A_L(\be) +\phi_0(0,\ga_0(s))\\
  \label{recta}
 S(0,r_T(t);1+1/t)+S(r_T(t),cs^{\frac 23};\si s) 
&\le 2t^{-\frac 13}A_L(\be)+S(0,cs^{\frac 23}).
\end{align}

\begin{proposition}\label{bounded}
There are constants $K,\bar t>0$ and $\bar s>1$ such that for every 
$t\ge\bar t$, $T\ge t\bar s$ we have $r_T(t)\le K$
\end{proposition}
\begin{proof}

Suppose the Proposition is false, then there are sequences $K_n\to\infty$,
$t_n\to\infty$, $T_n$ such that $s_n=T_n/t_n\to\infty$,
$r_n=r_{T_n}(t_n)\to\infty$. Note that $\si_n=\si(T_n,t_n)\to 1$.

If $r_n\le s_n^{\frac 13}$, inequality
\eqref{recta} and items (b) and (c) of Proposition \ref{potential-line} give
\[
\frac{r_n^2t_n}{2(1+t_n)}+o(r_n^2)+
(6U_0^2s_n)^{\frac 13}[\frac 59(\si_n-1)^2+o((\si_n-1)^2))]
-(8U_0r_n)^{\frac 12}+O(s_n^{-\frac 13}) \le 2t_n^{-\frac 13}A_L(\be)
\]
which is impossible for $n$ large.
If $r_n>s_n^{\frac 13}$, from item (b) of Proposition
\ref{potential-line} we have  
\begin{align*}
2t_n^{-\frac 13}A_L(\be)&\ge S(0,r_n;1+\frac 1t_n)-S(0,cs_n^{\frac 23})\\
&\ge \frac{s_n^{\frac 23}t_n}{2(1+t_n)}+o(s_n^{\frac 23})-2(6U_0^2s_n)^{\frac 13}
\end{align*}
which is impossible for $n$ large as well.
\end{proof}

\begin{lemma}\label{solucion}
Busemann function $u$ is a viscosity solution of the Hamilton-Jacobi equation
\eqref{eq:hj}  
\end{lemma}
\begin{proof} For $x\in(\R^d)^N$, let $y_T:[0,\tau_T]\to(\R^d)^N$ be
such that \eqref{eq:minimiza} holds.  % and $r_T(t)=y_T(t)t^{-\frac 23}$.

By Theorem \ref{Mad1}  and Proposition \ref{bounded} 
there are constants $K,\bar t>0$ and $\bar s>1$ such that for every 
$t\ge\bar t$, $T\ge t\bar s$ we have
\begin{equation}\label{cotaccion}
A_L(y_T|[0,t])=\phi(x,y_T(t);t)\leq K(t^{\frac 43}t^{-1}+tt^{-\frac 23})=2Kt^{\frac 13}.
\end{equation}
We claim that the family 
\begin{equation}\label{fam}
\{y_T|[0,t]\}
\end{equation}
is equicontinuous. Indeed, by \eqref{cotaccion} we have
$$
\int_0^{t}\|\dot{y_T}\|^2ds\leq2A_L(y_T|[0,t])\leq 4Kt^{\frac 13}.
$$
Thus, for each $0<s<s'\leq t$
$$
\|y_T(s)-y_T(s')\|\leq\int_s^{s'}\|\dot y_T(\up)\|d\up\leq\sqrt{s'-s}(\int_s^{s'}\|\dot y_T(\up)\|^2)^{\frac12}
\leq 2(Kt^{\frac 13})^{\frac12}\sqrt{s'-s},
$$
showing the equicontinuity of \eqref{fam}. Since $y_T(0)=x$ the
family is also equibounded. From Ascoli's Theorem, there is a sequence
$T^n\to\infty$ satifying $T^n\ge  t\bar s$ such that ${y_{T^n}|[0,t]}$  converges uniformly. 
Applying this argument to an increasing sequence $t_k\to\infty$, by a diagonal trick one 
obtains a sequence $T_n\to\infty$ such that $y_{T_n}\vert_{[0,t]}\rightarrow\af\vert_{[0,t]}$ 
uniformly for each $t>0$.  By lower semi-continuity
\begin{equation}
  \label{eq:lower}
  \liminf_{n\to\infty}A_L(y_{T_n}|_{[0,t]})\ge A_L(\af|_{[0,t]}).
\end{equation}
From Theorem \ref{Mad2} 
\begin{equation}
  \label{ineq}
  |\phi(\af(t),\ga(T))-\phi(y_T(t),\ga(T))|\le \phi(\af(t),y_T(t))
\le\eta\|\af(t)-y_T(t)\|^{\frac 12}.
\end{equation}
Using \eqref{1}, \eqref{eq:lower},\eqref{ineq} we have for all $t>0$
\begin{align*}
 u(\af(t))&=\lim_n\phi(\ga(0),\ga(T_n))-\phi(\af(t),\ga(T_n))\\
          &=\lim_n \phi(\ga(0),\ga(T_n))-\phi(x,\ga(T_n))+A_L(y_{T_n}|_{[0,t]})\\
	   &\ge u(x)+A_L(\af|_{[0,t]}) \ge u(\af(t)).
 \end{align*}
Then $\af$ calibrates $u$.
\end{proof}

We now address the proof of equality \eqref{asintotica}.
\begin{proposition}\label{cerca-norma}
For any $\ep>0$ there is $t_\ep>0$ such that for $t\ge t_\ep$, $T\ge t\bar s$ 
\begin{equation}
  \label{norma-cerca}
  |r_T(t)-c|<\ep.
\end{equation}
\end{proposition}
\begin{proof}
For $\xi:[0,1+\frac 1t]\to\R^+$ with action 
$A(\xi)=S(0,r_T(t);1+\frac 1t)$ consider the reparametrization
$\xi^*:[0,1]\to\R^+$ given by  $\xi^*(s)=\xi(s(1+t)/t)$. Then
\begin{align}\notag
 S(0,r_T(t);1)&\le A(\xi^*)=
\frac{(1+t)}t\frac 12\int_0^{(t+1)/t}|\dot\xi(s)|^2ds 
+\frac t{(1+t)}\int_0^{(t+1)/t}U(\xi(s))ds\\
&\le (1+\frac 1t) S(0,r_T(t);1+\frac 1t)\label{comparison}
\end{align}
Let $\bar t >0$, $\bar s>1$ be given by Proposition 1.
From inequalities \eqref{homo}, \eqref{cero}, \eqref{recta} and
\eqref{comparison}, there is constant $K_1$ such that for $t\ge\bar t$, 
$T\ge t\bar s$ we have 
\begin{align}\label{homo1}
  \phi(0,y_T(t)t^{-\frac 23};1)
+\phi(y_T(t)t^{-\frac 23},\ga_0(s);\si s) 
&\le\phi(0,\ga_0(s)) +K_1t^{-\frac 13}\\ 
\label{uno} \phi_0(0,y_T(t)t^{-\frac 23};1)
+\phi_0(y_T(t)t^{-\frac 23},\ga_0(s);\si s) 
&\le \phi_0(0,\ga_0(s)) +K_1t^{-\frac 13}\\
\label{recta1} S(0,r_T(t);1)+S(r_T(t),cs^{\frac 23};\si s)
&\le S(0,cs^{\frac 23}) +K_1t^{-\frac 13}
\end{align}
Inequality \eqref{recta1} and triangle inequality 
\begin{equation}\label{triangulo} 
S(0,u)\le S(0,r)+S(r,u)
\end{equation}
imply 
\[S(0,r_T(t);1)\le S(0,r_T(t)) +K_1t^{-\frac 13}.\]
Proposition \ref{cerca-norma} then follows from Proposition \ref{potential-line}(a). 
\end{proof}

\begin{proposition}\label{cerca}
 Given $\ep>0$ there are  $\bar t_\ep>0$, $\bar s_\ep>1$ such that for $t>\bar t_\ep$,
 $T\ge t\bar s_\ep$ we   have  $|r_T(t)-c|<\ep$ and $\angle(r_T(t),x_0)<\ep$
\end{proposition}
\begin{proof}
Consider the Kepler's problem on $(\R^d)^N$. 
Let $z$ be a configuration satisfying $|\|z\|-c|<\ep$ and
$\tau > 0$.  The minimizer (for $L_0$) $\xi:[0,\tau]\to(\R^d)^N$
joining $z$ to $\ga_0(s)$ in time $\tau$ is a collision-free Keplerian
arc; hence it is contained in the plane generated by $0, z$ and $\ga_0(s)$. Introducing 
polar coordinates in this plane, one can identify $z$ with $re^{i\te}$ and $\ga_0(s)$ 
with $cs^{2/3}\in\R\subset\C$, where $|r-c|<\ep$, $|\te|\le\pi$.

The path $\xi$ can be written in polar coordinates as 
\[\xi(\up)=\rho(\up)e^{i\om(\up)}, \quad u\in[0,\tau]\] where
\[\rho(0) = r,\quad \om(0) =\te\]
\[\rho(\tau) = cs^{2/3},\quad \om(\tau)\in 2\pi\Z.\]
We claim that $\xi$ is a direct path, that is to say, the total variation of the polar angle 
$\om$ is less than or equal to $\pi$. Assume for the contrary that $|\om(\tau)-\te|>\pi$. 
Changing the orientation of the plane if necessary, we can assume that
$\om(\tau)\ge 2\pi$; hence there exists a unique integer $k\ge 1$ 
such that $\om(\tau)=2k\pi$.

The path $\bar\xi(\up)=\rho(\up)e^{i\bar\om(\up)}$ with 
\[\bar\om(\up)=\te-\frac \te{2k\pi-\te}(\om(\up)-\te)\]
has the same ends as $\xi$ and 
\[A_{L_0}(\bar\xi)-A_{L_0}(\xi)=\frac 12
\left[\Big(\frac\te{2k\pi-\te}\Big)^2-1\right] 
\int_0^\tau(\rho^2\dot\om^2)(\up)d\up<0,\]
which is a contradiction.
Lambert's theorem (see \cite{A}), states that if $x_1$ and $x_2$ are two
configurations and  $\tau>0$, the action $A_{L_0}(x_1,x_2;\tau)$ of the
direct Keplerian arc joining $x_1$ to $x_2$ in time $\tau$ is a
function of three parameters only: the time $\tau$, the distance
$|x_1-x_2|$ between the two ends and the sum of the distances between
the ends and the origin (i.e. $|x_1| + |x_2|$). Thus
\begin{equation}
  \label{eq:lambert}
  \phi_0(re^{i\te},\ga_0(s);\tau)=S(d_1(r,\te,s),d_2(r,\te,s);\tau)
\end{equation}
where
\begin{align*}
  2d_1(r,\te,s)&=r+cs^{\frac 23}-|re^{i\te}-cs^{\frac 23}|
=r(1+\cos\te)-l(r,s,\te)\\
  2d_2(r,\te,s)&=r+cs^{\frac 23}+|re^{i\te}-cs^{\frac 23}|
=2cs^{\frac 23}+r(1-\cos\te)+l(r,s,\te)\\
      l(r,s,\te)&=O(s^{-\frac 23}), \quad s\to\infty
\end{align*}
uniformly for $r$ bounded.

Letting $\si^*=\Big(\dfrac c{d_2}\Big)^{\frac32}\si s$, from \eqref{eq:S(r,u)}
we have for $r$ bounded, $s$ large and $\si$ close to $1$
\begin{align*}
  S(d_1,d_2;\si s)&= (8U_0)^{\frac 12}
(d_2^{\frac 12}(1+\frac 5{18}(\si^*-1)^2+o((\si^*-1)^2))-d_1^{\frac12})
+O(s^{-\frac 13})\\
&\ge 2 (6U_0^2s)^{\frac 13}%(1+\frac 5{18}(\si-1)^2+o((\si-1)^2))\\&
-2(U_0r(1+\cos\te))^{\frac 12}+O(s^{-\frac 13}).
\end{align*}
From equation \eqref{eq:S(r,s)} for $\si=1$ we have
\[S(r,cs^{\frac 23})\le S(r,cs^{\frac 23};s)=
 2(6U_0^2s)^{\frac 13}-(8U_0r)^{\frac 12}+O(s^{-\frac 13}).\]
Thus, for $r$ bounded, $s$ large and $\si$ close to $1$, we have  
\begin{equation}\label{eq:1}
S(d_1,d_2;\si s)\ge S(r,cs^{\frac 23})
+2(U_0r(1-\cos\te))^{\frac 12}+O(s^{-\frac 13}).
\end{equation}
Recalling \eqref{remark2}, inequalities \eqref{uno},
\eqref{triangulo} imply for $s=T/t$, $\si=\si(T,t)$
\begin{equation}\label{eq:2}
\phi_0(y_T(t)t^{-\frac 23},\ga_0(s);\si s) 
\le 2t^{-\frac 13}A_L(\be)+ S(r_T(t),cs^{\frac  23})
\end{equation}
From \eqref{norma-cerca}, \eqref{eq:lambert}, \eqref{eq:1},
\eqref{eq:2} we have 
\[2t^{-\frac 13}A_L(\be)\ge 2(U_0r_T(t)(1-\cos\te_T(t)))^{\frac 12}
+O\big((T/t)^{-\frac 13}\big)\]
It follows that given $\ep>0$ there are $\bar t_\ep>0$, $\bar s_\ep>1$ 
such that for $t\ge\bar t_\ep, T\ge t\bar s_\ep$ we have
$|\te_T(t)|<\ep$.
\end{proof}
For $\bar t_\ep>0$, $\bar s_\ep>1$ as in Proposition \ref{cerca},
$t\ge\bar t_\ep, T\ge t\bar s_\ep$ ,
\[\|y_T(t)t^{-\frac 23}-cx_0\|\le 2\ep.\]
Taking $T=T_n$, $n\to\infty$ we have for $t\ge\bar t_\ep$
\[\|\af(t)t^{-\frac 23}-cx_0\|\le 2\ep.\]

\end{document}